\documentclass[11pt]{article}\UseRawInputEncoding
\usepackage{amssymb,amsfonts,amsmath,latexsym,epsf,tikz,url}
\usepackage{graphics}
\usepackage[usenames,dvipsnames]{pstricks}
\usepackage{pstricks-add}
\usepackage{epsfig}
\usepackage{pst-grad} % For gradients
\usepackage{pst-plot} % For axes
\usepackage[space]{grffile} % For spaces in paths
\usepackage{etoolbox} % For spaces in paths
\makeatletter % For spaces in paths
\patchcmd\Gread@eps{\@inputcheck#1 }{\@inputcheck"#1"\relax}{}{}
\makeatother

\newtheorem{theorem}{Theorem}[section]
\newtheorem{proposition}[theorem]{Proposition}

\newtheorem{corollary}[theorem]{Corollary}
\newtheorem{lemma}[theorem]{Lemma}
\newtheorem{remark}[theorem]{Remark}
\newtheorem{example}[theorem]{Example}

\newcommand{\qed}{\hfill $\square$\medskip}

\begin{document}

\title{More on co-even domination number}

\author{
Nima Ghanbari
}

\date{\today}

\maketitle

\begin{center}
Department of Informatics, University of Bergen, P.O. Box 7803, 5020 Bergen, Norway\\
\bigskip

{\tt Nima.Ghanbari@uib.no }
\end{center}

%%%%%%%%%%%%%%ABSTRACT%%%%%%%%%%%%%%%%%%%%%%%%%%%%%%%%%%%%%%%%%%%%%%%%%%%%%%%%%%%%

\begin{abstract}
Let $G=(V,E)$ be a simple graph. A dominating set of $G$ is a subset $D\subseteq V$ such that every vertex not in $D$ is adjacent to at least one vertex in $D$.
The cardinality of a smallest dominating set of $G$, denoted by $\gamma(G)$, is the domination number of $G$. A dominating set $D$ is called co-even dominating set if the degree of vertex $v$ is even number for all $v\in V-D$. 
 The cardinality of a smallest co-even dominating set of $G$, denoted by $\gamma _{coe}(G)$, is the co-even domination number of $G$.
In this paper, we find more results on co-even domination number of  graphs and  count the  number of co-even dominating sets of some specific graphs.
\end{abstract}

\noindent{\bf Keywords:} domination number, co-even dominating set, subdivision, path.

\medskip
\noindent{\bf AMS Subj.\ Class.:} 05C38, 05C69,  05C75 ,  05C76

%%%%%%%%%%%%%%%%%%%%%%%%%%%%%%%%%%%%%%%%%%%%%%%%%%%%%%%%%%%%%%%%%%%%%%%%%%%%%%%%%
%%%%%%%%%%%%%%%%%%%%%%%%%%%%%%%%%%%%%%%%%%%%%%%%%%%%%%%%%%%%%%%%%%%%%%%%%%%%%%%%%
\section{Introduction}

Let $G = (V,E)$ be a  graph with vertex set $V$ and edge set $E$. Throughout this paper, we consider only simple graphs.
For each vertex $v\in V$, the set $N_G(v)=\{u\in V | uv \in E\}$ refers to the open neighbourhood of $v$ and the set $N_G[v]=N_G(v)\cup \{v\}$ refers to the closed neighbourhood of $v$ in $G$. The degree of $v$, denoted by $deg(v)$, is the cardinality of $N_G(v)$.  A set $S\subseteq V$ is a  dominating set if every vertex in $V- S$ is adjacent to at least one vertex in $S$.
The  domination number $\gamma(G)$ is the minimum cardinality of a dominating set in $G$. 
 We say that $u$ in dominating set $S$ dominates $v$, if $v\in N_G(u)$. See \cite{domination} for more detail in domination number.

There are various domination numbers in the literature. One of them is co-even domination number.
A dominating set $D$ is called a co-even dominating set if the degree of vertex $v$ is even number for all $v\in V-D$. The cardinality of a smallest co-even dominating set of $G$, denoted by $\gamma _{coe}(G)$, is the co-even domination number of $G$ \cite{Sha}. We refer the reader to \cite{Demir, Nima, Nima1, Sha1} for more detail in co-even domination number of a graph. 

We denote by $G\odot v$ the graph obtained from $G$ by the removal of all edges between
 any pair of neighbours of $v$, note $v$ is not removed from the graph \cite{alikhani1}. The \textit{ $k$-subdivision} of $G$, denoted by $G^{\frac{1}{k}}$, is constructed by replacing each edge $v_iv_j$ of $G$ with a path of length $k$, say $P^{\{v_i,v_j\}}$. These $k$-paths are called \textit{superedges}, any new vertex is an internal vertex, and is denoted by $x^{\{v_i,v_j\}}_l$ if it belongs to the superedge $P_{\{v_i,v_j\}}$, $i<j$, with  distance $l$ from the vertex $v_i$, where $l \in \{1, 2, \ldots , k-1\}$.  Note that for $k = 1$, we have $G^{1/1}= G^1 = G$, and if the graph $G$ has $v$ vertices and $e$ edges, then the graph $G^{\frac{1}{k}}$ has $v+(k-1)e$ vertices and $ke$ edges. Some results about subdivision of a graph can be found in \cite{ALikhani12,Babu,Nima2}.
 
 The concept of domination and related invariants have
been generalized in many ways.  Most of the papers published so far deal with structural
aspects of domination, trying to determine exact expressions for $\gamma(G)$  or some upper and/or lower bounds for it. There were no paper concerned with the 
enumerative side of the problem by 2008.  
Regarding to enumerative side of dominating sets, domination polynomial of graph has introduced in \cite{saeid1}. The domination polynomial of graph $G$ is the  generating function for the number of dominating sets of  $G$, i.e., $D(G,x)=\sum_{ i=1}^{|V(G)|} d(G,i) x^{i}$ (see \cite{euro,saeid1}).   This  polynomial and its roots has been actively studied in recent
years (see for example \cite{Kot,Oboudi}). 
It is natural to count the number of another kind of dominating sets (\cite{utilitas,Doslic,Nima3}).     
Let ${\cal D}_{coe}(G,i)$ be the family of
co-even  dominating sets of a graph $G$ with cardinality $i$ and let
$d_{coe}(G,i)=|{\cal D}_{coe}(G,i)|$. The generating function for the number of co-even dominating sets of $G$ is $D_{coe}(G,x)=\sum_{i=1}d_0(G,i)x^i$.

In the next Section, we study the co-even domination number of graphs constructed by the removal of all edges between
 any pair of neighbours of a vertex. In Section 3, we find upper bound for co-even domination number of $k$-subdivision of graphs. Finally, in Section 4, we find the number of co-even dominating sets of some specific graphs and their generating functions.

\section{Co-even domination number of $G\odot v$}

$G\odot v$ is the graph obtained from $G$ by the removal of all edges between
 any pair of neighbours of $v$. In this section we obtain sharp upper and lower bounds
 for the co-even domination number of $G\odot v$.
First we state some known results.

	\begin{proposition}\cite{Sha}\label{pro-sha}
Let $G=(V,E)$ be a graph and $D$ is a co-even dominating set. Then,
\begin{itemize}
\item[(i)]
All vertices of odd or zero degrees belong to every co-even dominating set.
\item[(ii)]
$deg(v)\geq 2$, for all $v\in V-D$.
\item[(iii)]
$\gamma (G) \leq \gamma_{coe} (G).$
\end{itemize}
	\end{proposition}

Now we consider to $G\odot v$ and find upper and lower bounds
 for the co-even domination number of that.

\begin{theorem}\label{Godotv}
Let $G=(V,E)$ be a graph and $v\in V$. Then,
$$\gamma _{coe}(G) - deg(v) +1 \leq \gamma _{coe}(G\odot v)\leq \gamma _{coe}(G) + deg(v) -1.$$
\end{theorem}

\begin{proof}
Suppose that $v\in V$ and $D_{coe}(G)$ is co-even dominating set of $G$. First we find the upper bound for $\gamma _{coe}(G\odot v)$. We consider the following cases:
\begin{itemize}
\item[($i$)]
$v\notin D_{coe}(G)$. In this case, $deg(v)$ should be even and there exists $u\in N_G(v)$ such that $u\in D_{coe}(G)$. By the definition of $G\odot v$, removal all edges between any pair of neighbours of $v$, may change their degree to an odd number but the degree of $v$ does not change and remains even. So 
$$D_{coe}(G)\cup N_G(v)$$
 is a co-even dominating set for $G\odot v$, and since $u$ is adjacent to $v$ and $u\in D_{coe}(G)$, then $\gamma _{coe}(G\odot v)\leq \gamma _{coe}(G) + deg(v) -1$.
\item[($ii$)]
$v\in D_{coe}(G)$ and $deg(v)$ is even. In this case we consider to the set 
$$\left( D_{coe}(G)-\{v\} \right) \cup N_G(v).$$
 It is easy to see that this set is a co-even dominating set for $G\odot v$ and since $deg(v)$ is even, we do not need it in our dominating set and it dominates with all of its neighbours. So  $\gamma _{coe}(G\odot v)\leq \gamma _{coe}(G) + deg(v) -1$.
\item[($iii$)] 
$v\in D_{coe}(G)$ and $deg(v)$ is odd. First, suppose that there is at least a vertex in $N_G(v)\cap D_{coe}(G)$. Then $D_{coe}(G) \cup N_G(v)$ is a co-even dominating set for $G\odot v$ with size at most $\gamma _{coe}(G) + deg(v) -1$.
Now, suppose that $N_G(v)\cap D_{coe}(G)=\{\} $.
We show that all vertices in $ N_G(v)$ can not have even degree and connected to odd number of vertices in $ N_G(v)$ at the same time. Suppose that this happens and we have all vertices in $ N_G(v)$ with even degree and they are connected to odd number of vertices in $ N_G(v)$ at the same time. By removing $v$ and all vertices in $G$ which are not in $ N_G(v)$ and consider to this subgraph of $G$, then we have a graph with odd number of vertices and all of the vertices have odd degree, which is a contradiction to the Handshake lemma. So this does not happen. Therefore, there exists at least a vertex $w\in N_G(v)$ which is adjacent to even number of vertices in $ N_G(v)$. Note that  $deg(w)$ remains even in  $G\odot v$. By considering 
$$\left( N_G(v) -\{w\} \right) \cup D_{coe}(G) $$
 as our dominating set, we have a co-even dominating set with size $\gamma _{coe}(G) + deg(v) -1$. Hence, $\gamma _{coe}(G\odot v)\leq \gamma _{coe}(G) + deg(v) -1$.
\end{itemize}

Now we find the lower bound for $\gamma _{coe}(G\odot v)$. First we form $G\odot v$ and find a co-even dominating set for it. Let $D_{coe}(G\odot v)$ be this set. We consider the following cases:

\begin{itemize}
\item[($i'$)]
$v\notin D_{coe}(G\odot v)$. In this case, $deg(v)$ should be even and there exists $u\in N_G(v)\cap D_{coe}(G\odot v)$. By considering 
$$D_{coe}(G\odot v)\cup N_G(v)$$
as our dominating set, we have  a co-even dominating set for $G$, and since $u$ is adjacent to $v$ and $u\in D_{coe}(G\odot v)$, then $\gamma _{coe}(G)\leq \gamma _{coe}(G\odot v) + deg(v) -1$.
\item[($ii'$)]
$v\in D_{coe}(G\odot v)$ and $deg(v)$ is even. In this case we consider to the set 
$$\left( D_{coe}(G\odot v)-\{v\} \right) \cup N_G(v).$$
 It is easy to see that this set is a co-even dominating set for $G$ and since $deg(v)$ is even, we do not need it in our dominating set. So  $\gamma _{coe}(G)\leq \gamma _{coe}(G\odot v) + deg(v) -1$.
\item[($iii'$)] 
$v\in D_{coe}(G\odot v)$ and $deg(v)$ is odd. First, suppose that there is at least a vertex in $N_G(v)\cap D_{coe}(G\odot v)$. Then $D_{coe}(G\odot v) \cup N_G(v)$ is a co-even dominating set for $G$ with size at most $\gamma _{coe}(G) + deg(v) -1$.
Now, suppose that $N_G(v)\cap D_{coe}(G\odot v)=\{\} $.
By the same argument as case ($iii$),  we conclude
that all vertices in $ N_G(v)$ can not have even degree and connected to odd number of vertices in $ N_G(v)$ at the same time. Therefore, there exists at least a vertex $w\in N_G(v)$ which is adjacent to even number of vertices in $ N_G(v)$. Now by considering 
$$\left( N_G(v) -\{w\} \right) \cup D_{coe}(G\odot v) $$
 as our dominating set, we have a co-even dominating set with size $\gamma _{coe}(G\odot v) + deg(v) -1$ for $G$. Hence, $\gamma _{coe}(G)\leq \gamma _{coe}(G\odot v) + deg(v) -1$.
\end{itemize}
Therefore we have the result.
\qed
\end{proof}

We end this section by showing that bounds for co-even domination number of $G\odot v$ are sharp:

\begin{remark}
The  bounds in Theorem \ref{Godotv} are sharp. For the upper bound, it suffices to consider $G$ as shown in Figure \ref{Godotvuppersharp}. The set of black vertices in $G$ is a co-even dominating set of $G$. Also, the set of black vertices is a co-even dominating set of $G\odot v$, and $\gamma _{coe}(G\odot v)= \gamma _{coe}(G) + deg(v) -1$. 
For the lower bound, 
it suffices to consider $H$ as shown in Figure \ref{Godotvlowersharp}. The set of black vertices in $H$ and $H\odot v$ are co-even dominating sets of them, respectively. So $\gamma _{coe}(H\odot v)= \gamma _{coe}(H) - deg(v) +1$.
\end{remark}

	\begin{figure}
		\begin{center}
			\psscalebox{0.5 0.5}
{
\begin{pspicture}(0,-8.131442)(14.5971155,1.3256732)
\psdots[linecolor=black, dotsize=0.4](0.4,-2.4714422)
\psline[linecolor=black, linewidth=0.08](0.4,-2.4714422)(2.4,-0.07144226)(2.4,-0.07144226)
\psline[linecolor=black, linewidth=0.08](0.4,-2.4714422)(4.0,-1.2714423)(4.0,-1.2714423)
\psline[linecolor=black, linewidth=0.08](0.4,-2.4714422)(4.0,-3.2714422)(4.0,-3.2714422)
\psline[linecolor=black, linewidth=0.08](0.4,-2.4714422)(2.4,-4.471442)(2.4,-4.471442)
\psline[linecolor=black, linewidth=0.08](0.4,-2.4714422)(2.4,-6.8714423)(2.4,-6.8714423)
\psline[linecolor=black, linewidth=0.08](2.4,-0.07144226)(4.0,-1.2714423)(4.0,-1.2714423)
\psline[linecolor=black, linewidth=0.08](2.4,-4.471442)(4.0,-3.2714422)(4.0,-3.2714422)
\psline[linecolor=black, linewidth=0.08](4.0,-1.2714423)(4.0,-3.2714422)(4.0,-3.2714422)
\psline[linecolor=black, linewidth=0.08](2.4,-0.07144226)(2.4,-4.471442)(2.4,-4.471442)
\psline[linecolor=black, linewidth=0.08](2.4,-0.07144226)(4.0,-3.2714422)(4.0,-3.2714422)
\psline[linecolor=black, linewidth=0.08](2.4,-4.471442)(4.0,-1.2714423)(4.0,-1.2714423)
\psline[linecolor=black, linewidth=0.08](2.4,-4.471442)(5.2,-4.471442)(4.8,-4.471442)
\psline[linecolor=black, linewidth=0.08](2.4,-4.471442)(5.2,-5.671442)(5.2,-5.671442)
\psline[linecolor=black, linewidth=0.08](2.4,-0.07144226)(5.2,-0.07144226)
\psline[linecolor=black, linewidth=0.08](2.4,-0.07144226)(5.2,1.1285577)(5.2,1.1285577)
\psline[linecolor=black, linewidth=0.08](4.0,-1.2714423)(5.2,-0.87144226)(5.2,-0.87144226)
\psline[linecolor=black, linewidth=0.08](4.0,-1.2714423)(5.2,-1.6714423)(5.2,-1.6714423)
\psline[linecolor=black, linewidth=0.08](4.0,-3.2714422)(5.2,-2.8714423)(5.2,-2.8714423)
\psline[linecolor=black, linewidth=0.08](4.0,-3.2714422)(5.2,-3.6714423)(5.2,-3.6714423)
\psline[linecolor=black, linewidth=0.08](2.4,-6.8714423)(5.2,-6.8714423)(5.2,-6.8714423)
\psdots[linecolor=black, dotsize=0.4](5.2,1.1285577)
\psdots[linecolor=black, dotsize=0.4](5.2,-0.07144226)
\psdots[linecolor=black, dotsize=0.4](5.2,-0.87144226)
\psdots[linecolor=black, dotsize=0.4](5.2,-1.6714423)
\psdots[linecolor=black, dotsize=0.4](5.2,-2.8714423)
\psdots[linecolor=black, dotsize=0.4](5.2,-4.471442)
\psdots[linecolor=black, dotsize=0.4](5.2,-5.671442)
\psdots[linecolor=black, dotsize=0.4](5.2,-6.8714423)
\rput[bl](0.0,-2.8714423){$v$}
\psdots[linecolor=black, dotsize=0.4](9.6,-2.4714422)
\psdots[linecolor=black, dotsize=0.4](13.2,-1.2714423)
\psdots[linecolor=black, dotsize=0.4](13.2,-3.2714422)
\psdots[linecolor=black, dotsize=0.4](11.6,-4.471442)
\psdots[linecolor=black, dotsize=0.4](11.6,-0.07144226)
\psline[linecolor=black, linewidth=0.08](9.6,-2.4714422)(11.6,-0.07144226)(11.6,-0.07144226)
\psline[linecolor=black, linewidth=0.08](9.6,-2.4714422)(13.2,-1.2714423)(13.2,-1.2714423)
\psline[linecolor=black, linewidth=0.08](9.6,-2.4714422)(13.2,-3.2714422)(13.2,-3.2714422)
\psline[linecolor=black, linewidth=0.08](9.6,-2.4714422)(11.6,-4.471442)(11.6,-4.471442)
\psline[linecolor=black, linewidth=0.08](9.6,-2.4714422)(11.6,-6.8714423)(11.6,-6.8714423)
\psdots[linecolor=black, dotsize=0.4](13.2,-1.2714423)
\psline[linecolor=black, linewidth=0.08](11.6,-4.471442)(14.4,-4.471442)(14.0,-4.471442)
\psline[linecolor=black, linewidth=0.08](11.6,-4.471442)(14.4,-5.671442)(14.4,-5.671442)
\psline[linecolor=black, linewidth=0.08](11.6,-0.07144226)(14.4,-0.07144226)
\psline[linecolor=black, linewidth=0.08](11.6,-0.07144226)(14.4,1.1285577)(14.4,1.1285577)
\psline[linecolor=black, linewidth=0.08](13.2,-1.2714423)(14.4,-0.87144226)(14.4,-0.87144226)
\psline[linecolor=black, linewidth=0.08](13.2,-1.2714423)(14.4,-1.6714423)(14.4,-1.6714423)
\psline[linecolor=black, linewidth=0.08](13.2,-3.2714422)(14.4,-2.8714423)(14.4,-2.8714423)
\psline[linecolor=black, linewidth=0.08](13.2,-3.2714422)(14.4,-3.6714423)(14.4,-3.6714423)
\psline[linecolor=black, linewidth=0.08](11.6,-6.8714423)(14.4,-6.8714423)(14.4,-6.8714423)
\psdots[linecolor=black, dotsize=0.4](14.4,1.1285577)
\psdots[linecolor=black, dotsize=0.4](14.4,-0.07144226)
\psdots[linecolor=black, dotsize=0.4](14.4,-0.87144226)
\psdots[linecolor=black, dotsize=0.4](14.4,-1.6714423)
\psdots[linecolor=black, dotsize=0.4](14.4,-2.8714423)
\psdots[linecolor=black, dotsize=0.4](14.4,-4.471442)
\psdots[linecolor=black, dotsize=0.4](14.4,-5.671442)
\psdots[linecolor=black, dotsize=0.4](14.4,-6.8714423)
\rput[bl](9.2,-2.8714423){$v$}
\psdots[linecolor=black, dotstyle=o, dotsize=0.4, fillcolor=white](2.4,-0.07144226)
\psdots[linecolor=black, dotstyle=o, dotsize=0.4, fillcolor=white](4.0,-1.2714423)
\psdots[linecolor=black, dotstyle=o, dotsize=0.4, fillcolor=white](4.0,-3.2714422)
\psdots[linecolor=black, dotstyle=o, dotsize=0.4, fillcolor=white](2.4,-4.471442)
\psdots[linecolor=black, dotstyle=o, dotsize=0.4, fillcolor=white](2.4,-6.8714423)
\psdots[linecolor=black, dotstyle=o, dotsize=0.4, fillcolor=white](11.6,-6.8714423)
\rput[bl](12.0,-8.131442){\LARGE{$G\odot v$}}
\rput[bl](3.2,-8.071443){\LARGE{$G$}}
\psdots[linecolor=black, dotsize=0.4](5.2,-3.6714423)
\psdots[linecolor=black, dotsize=0.4](14.4,-3.6714423)
\end{pspicture}
}
		\end{center}
		\caption{Graphs $G$ and $G\odot v$} \label{Godotvuppersharp}
	\end{figure}
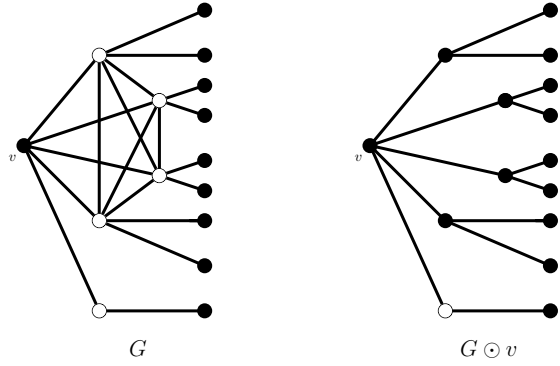

	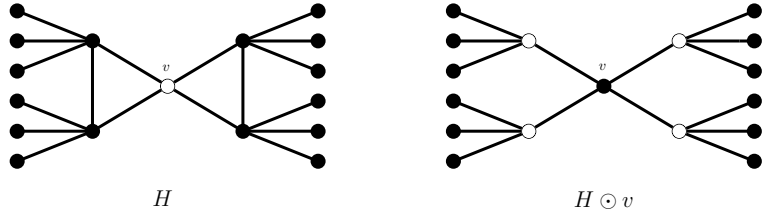
\begin{figure}
		\begin{center}
			\psscalebox{0.5 0.5}
{
\begin{pspicture}(0,-6.531442)(20.014818,-1.0743268)
\rput[bl](4.077704,-2.8714423){$v$}
\rput[bl](3.817704,-6.471442){\LARGE{$H$}}
\rput[bl](15.017704,-6.531442){\LARGE{$H\odot v$}}
\psline[linecolor=black, linewidth=0.08](2.2177038,-2.0714424)(4.217704,-3.2714422)(2.2177038,-4.471442)(2.2177038,-4.471442)
\psline[linecolor=black, linewidth=0.08](2.2177038,-2.0714424)(2.2177038,-4.471442)
\psline[linecolor=black, linewidth=0.08](6.217704,-2.0714424)(4.217704,-3.2714422)(6.217704,-4.471442)
\psline[linecolor=black, linewidth=0.08](6.217704,-2.0714424)(6.217704,-4.471442)(6.217704,-4.471442)
\psdots[linecolor=black, dotsize=0.4](8.217704,-1.2714423)
\psdots[linecolor=black, dotsize=0.4](8.217704,-2.8714423)
\psdots[linecolor=black, dotsize=0.4](8.217704,-3.6714423)
\psdots[linecolor=black, dotsize=0.4](8.217704,-5.2714424)
\psdots[linecolor=black, dotsize=0.4](0.21770386,-5.2714424)
\psdots[linecolor=black, dotsize=0.4](0.21770386,-3.6714423)
\psdots[linecolor=black, dotsize=0.4](0.21770386,-2.8714423)
\psdots[linecolor=black, dotsize=0.4](0.21770386,-1.2714423)
\psline[linecolor=black, linewidth=0.08](0.21770386,-1.2714423)(2.2177038,-2.0714424)(0.21770386,-2.8714423)(0.21770386,-2.8714423)
\psline[linecolor=black, linewidth=0.08](0.21770386,-3.6714423)(2.2177038,-4.471442)(2.2177038,-4.471442)
\psline[linecolor=black, linewidth=0.08](2.2177038,-4.471442)(0.21770386,-5.2714424)(0.21770386,-5.2714424)
\psline[linecolor=black, linewidth=0.08](6.217704,-2.0714424)(8.217704,-1.2714423)(8.217704,-1.2714423)
\psline[linecolor=black, linewidth=0.08](6.217704,-2.0714424)(8.217704,-2.8714423)(8.217704,-2.8714423)
\psline[linecolor=black, linewidth=0.08](6.217704,-4.471442)(8.217704,-3.6714423)(8.217704,-3.6714423)
\psline[linecolor=black, linewidth=0.08](6.217704,-4.471442)(8.217704,-5.2714424)(8.217704,-5.2714424)
\psline[linecolor=black, linewidth=0.08](13.817704,-2.0714424)(15.817704,-3.2714422)(13.817704,-4.471442)(13.817704,-4.471442)
\psline[linecolor=black, linewidth=0.08](17.817703,-2.0714424)(15.817704,-3.2714422)(17.817703,-4.471442)
\psdots[linecolor=black, dotsize=0.4](19.817703,-1.2714423)
\psdots[linecolor=black, dotsize=0.4](19.817703,-2.8714423)
\psdots[linecolor=black, dotsize=0.4](19.817703,-3.6714423)
\psdots[linecolor=black, dotsize=0.4](19.817703,-5.2714424)
\psdots[linecolor=black, dotsize=0.4](11.817704,-5.2714424)
\psdots[linecolor=black, dotsize=0.4](11.817704,-3.6714423)
\psdots[linecolor=black, dotsize=0.4](11.817704,-2.8714423)
\psdots[linecolor=black, dotsize=0.4](11.817704,-1.2714423)
\psline[linecolor=black, linewidth=0.08](11.817704,-1.2714423)(13.817704,-2.0714424)(11.817704,-2.8714423)(11.817704,-2.8714423)
\psline[linecolor=black, linewidth=0.08](11.817704,-3.6714423)(13.817704,-4.471442)(13.817704,-4.471442)
\psline[linecolor=black, linewidth=0.08](13.817704,-4.471442)(11.817704,-5.2714424)(11.817704,-5.2714424)
\psline[linecolor=black, linewidth=0.08](17.817703,-2.0714424)(19.817703,-1.2714423)(19.817703,-1.2714423)
\psline[linecolor=black, linewidth=0.08](17.817703,-2.0714424)(19.817703,-2.8714423)(19.817703,-2.8714423)
\psline[linecolor=black, linewidth=0.08](17.817703,-4.471442)(19.817703,-3.6714423)(19.817703,-3.6714423)
\psline[linecolor=black, linewidth=0.08](17.817703,-4.471442)(19.817703,-5.2714424)(19.817703,-5.2714424)
\rput[bl](15.677704,-2.8314424){$v$}
\psline[linecolor=black, linewidth=0.08](2.2177038,-2.0714424)(0.21770386,-2.0714424)
\psline[linecolor=black, linewidth=0.08](2.2177038,-4.471442)(0.21770386,-4.471442)
\psline[linecolor=black, linewidth=0.08](6.217704,-2.0714424)(8.217704,-2.0714424)
\psline[linecolor=black, linewidth=0.08](6.217704,-4.471442)(8.217704,-4.471442)(8.217704,-4.471442)
\psline[linecolor=black, linewidth=0.08](13.817704,-2.0714424)(11.817704,-2.0714424)(11.817704,-2.0714424)
\psline[linecolor=black, linewidth=0.08](13.817704,-4.471442)(11.817704,-4.471442)
\psline[linecolor=black, linewidth=0.08](17.817703,-2.0714424)(19.417704,-2.0714424)
\psline[linecolor=black, linewidth=0.08](19.417704,-2.0714424)(19.817703,-2.0714424)
\psline[linecolor=black, linewidth=0.08](17.817703,-4.471442)(19.817703,-4.471442)
\psdots[linecolor=black, dotsize=0.4](19.817703,-2.0714424)
\psdots[linecolor=black, dotsize=0.4](19.817703,-4.471442)
\psdots[linecolor=black, dotsize=0.4](8.217704,-4.471442)
\psdots[linecolor=black, dotsize=0.4](8.217704,-2.0714424)
\psdots[linecolor=black, dotsize=0.4](6.217704,-2.0714424)
\psdots[linecolor=black, dotsize=0.4](6.217704,-4.471442)
\psdots[linecolor=black, dotsize=0.4](2.2177038,-2.0714424)
\psdots[linecolor=black, dotsize=0.4](2.2177038,-4.471442)
\psdots[linecolor=black, dotsize=0.4](0.21770386,-4.471442)
\psdots[linecolor=black, dotsize=0.4](0.21770386,-2.0714424)
\psdots[linecolor=black, dotsize=0.4](11.817704,-4.471442)
\psdots[linecolor=black, dotsize=0.4](11.817704,-2.0714424)
\psdots[linecolor=black, dotsize=0.4](15.817704,-3.2714422)
\psdots[linecolor=black, dotstyle=o, dotsize=0.4, fillcolor=white](4.217704,-3.2714422)
\psdots[linecolor=black, dotstyle=o, dotsize=0.4, fillcolor=white](13.817704,-2.0714424)
\psdots[linecolor=black, dotstyle=o, dotsize=0.4, fillcolor=white](13.817704,-4.471442)
\psdots[linecolor=black, dotstyle=o, dotsize=0.4, fillcolor=white](17.817703,-2.0714424)
\psdots[linecolor=black, dotstyle=o, dotsize=0.4, fillcolor=white](17.817703,-4.471442)
\end{pspicture}
}
		\end{center}
		\caption{Graphs $H$ and $H\odot v$} \label{Godotvlowersharp}
	\end{figure}

\section{Co-even domination number of $k$-subdivision of a graph}

In this section, we study the co-even domination number of $k$-subdivision of  a graph and find upper bound for that. 
 The $k$-subdivision of $G$, denoted by $G^{\frac{1}{k}}$, is constructed by replacing each edge $v_iv_j$ of $G$ with a path of length $k$, say $P^{\{v_i,v_j\}}$. These $k$-paths are called \textit{superedges}, any new vertex is an internal vertex, and is denoted by $x^{\{v_i,v_j\}}_l$ if it belongs to the superedge $P_{\{v_i,v_j\}}$, $i<j$, with  distance $l$ from the vertex $v_i$, where $l \in \{1, 2, \ldots , k-1\}$. In the following, we present upper bound for co-even domination number of $k$-subdivision of a graph.
First we state a known result.

\begin{lemma}\label{lemma}\cite{Char}
$\gamma (P_n)=\lceil \frac{n}{3} \rceil$.
\end{lemma}

For study the co-even domination number of $k$-subdivision of  a graph, first we consider to $G^{\frac{1}{2}}$ and find a sharp upper bound for it.

	\begin{theorem}\label{G12}
Let $G=(V(G),E(G))$  be a graph. Then,
 $$\gamma _{coe}(G^{\frac{1}{2}})\leq  | V(G)|.$$
	\end{theorem}

\begin{proof}
By considering $V(G)$ in $G^{\frac{1}{2}}$, we have a dominating set for that. By considering any edge $uv\in E$, Since any vertex $x_1^{\{u,v\}}$ in  superedge $P^{\{u,v\}}$ in $G^{\frac{1}{2}}$ has degree 2, then we do not need them in our co-even dominating set if they dominate with some other vertices. Hence $V(G)$ is a co-even dominating set for  $G^{\frac{1}{2}}$ and we have the result.
\qed
\end{proof}

\begin{remark}
The  upper bound in Theorem \ref{G12} is sharp.  It suffices to consider $G$ as star graph $S_4=k_{1,3}$. One can easily check that $\gamma _{coe}(S_4^{\frac{1}{2}})=4$ which is the number of vertices of $S_4$.
\end{remark}

The following example shows that $\gamma _{coe}(G^{\frac{1}{2}})$ can be less than $| V(G)|$:

\begin{example}
For the path graph $P_4$, we have $P_4^{\frac{1}{2}}=P_7$, and it is easy to see that $\gamma _{coe}(P_4^{\frac{1}{2}})=3<4$.
\end{example}

Now we find the exact value of co-even domination number of $3$-subdivision of  a graph.

\begin{theorem}
Let $G=(V(G),E(G))$  be a graph. Then,
 $$\gamma _{coe}(G^{\frac{1}{3}})=   |V(G)|.$$
\end{theorem}

\begin{proof}
For every edge $uv\in E(G)$, we have a superedge $P^{\{u,v\}}$ with vertices $u$, $x_1^{\{u,v\}}$, $x_2^{\{u,v\}}$ and $v$ in $G^{\frac{1}{3}}$. To have a dominating set in four consecutive vertices, we need at least two vertices. By choosing $u$ and $v$, then we have a dominating set for $G^{\frac{1}{3}}$ with size $|V(G)|$. It is easy to see that there is no set with smaller size as dominating set for this graph. Since $x_1^{\{u,v\}}$ and $x_2^{\{u,v\}} $ in superedge $P^{\{u,v\}}$ have degree 2, then our dominating set, is a co-even dominating set too. Therefore $\gamma _{coe}(G^{\frac{1}{3}})=   |V(G)|$.
\qed
\end{proof}

Now we consider to $G^{\frac{1}{n}}$, where $n\geq 4$ and find an upper bound for  co-even domination number of that.

\begin{theorem}\label{G1n}
Let $G=(V(G),E(G))$ be a graph and $n\geq 4$. Then,
 $$\gamma _{coe}(G^{\frac{1}{n}})\leq |V(G)| + \gamma(P_{n-3})|E(G)|  .$$
\end{theorem}

\begin{proof}
 For every $uv\in E(G)$, we consider $u$, $x_1^{\{u,v\}}$, $x_2^{\{u,v\}}$, $\ldots$, $x_{n-2}^{\{u,v\}}$, $x_{n-1}^{\{u,v\}}$, $v$ as vertex sequence of $P^{\{u,v\}}$, as shown in Figure \ref{edge1nodd}. First we put all vertices in $V(G)$ in our dominating set. So we do not need $x_1^{\{u,v\}}$ and $x_{n-1}^{\{u,v\}}$ in our (co-even) dominating set since they dominate by vertices in $V(G) $. Now we have a path of order $n-3$ by vertices  $x_2^{\{u,v\}}$, $x_3^{\{u,v\}}$, $\ldots$, $x_{n-2}^{\{u,v\}}$ in $P^{\{u,v\}}$. We choose $\gamma(P_{n-3})$ vertices among these vertices which actually makes a dominating set for this path, and put these vertices in our dominating set. Then by applying this process to all edges, we have a co-even dominating set with size $|V(G)| + \gamma(P_{n-3})|E(G)|$ for $G^{\frac{1}{n}}$. Note that since we chose every vertex in $V(G)$, then we do not have any vertex with odd degree among vertices we did not choose. Therefore we
have the result.
\qed
\end{proof}

	\begin{figure}
		\begin{center}
			\psscalebox{0.5 0.5}
{
\begin{pspicture}(0,-4.8)(12.394231,-2.3)
\psdots[linecolor=black, dotsize=0.4](0.19711533,-3.55)
\psdots[linecolor=black, dotsize=0.4](12.197115,-3.55)
\psline[linecolor=black, linewidth=0.08](0.19711533,-3.55)(6.997115,-3.55)(6.997115,-3.55)
\psline[linecolor=black, linewidth=0.08](8.5971155,-3.55)(12.197115,-3.55)(12.197115,-3.55)
\psdots[linecolor=black, fillstyle=solid, dotstyle=o, dotsize=0.4, fillcolor=white](10.5971155,-3.55)
\psdots[linecolor=black, fillstyle=solid, dotstyle=o, dotsize=0.4, fillcolor=white](8.997115,-3.55)
\psdots[linecolor=black, fillstyle=solid, dotstyle=o, dotsize=0.4, fillcolor=white](6.5971155,-3.55)
\psdots[linecolor=black, fillstyle=solid, dotstyle=o, dotsize=0.4, fillcolor=white](4.997115,-3.55)
\psdots[linecolor=black, fillstyle=solid, dotstyle=o, dotsize=0.4, fillcolor=white](3.3971152,-3.55)
\psdots[linecolor=black, fillstyle=solid, dotstyle=o, dotsize=0.4, fillcolor=white](1.7971153,-3.55)
\psdots[linecolor=black, dotsize=0.1](7.397115,-3.55)
\psdots[linecolor=black, dotsize=0.1](7.7971153,-3.55)
\psdots[linecolor=black, dotsize=0.1](8.197115,-3.55)
\rput[bl](0.037115324,-4.33){$u$}
\rput[bl](12.057116,-4.39){$v$}
\rput[bl](1.2571154,-4.39){$x_1^{\{u,v\}}$}
\rput[bl](2.9971154,-4.39){$x_2^{\{u,v\}}$}
\rput[bl](4.437115,-4.47){$x_3^{\{u,v\}}$}
\rput[bl](6.0971155,-4.43){$x_4^{\{u,v\}}$}
\rput[bl](8.497115,-4.43){$x_{n-2}^{\{u,v\}}$}
\rput[bl](10.1371155,-4.47){$x_{n-1}^{\{u,v\}}$}
\psbezier[linecolor=blue, linewidth=0.08, linestyle=dotted, dotsep=0.10583334cm](1.7971153,-2.35)(2.1971154,-2.35)(3.3571153,-3.73)(1.7971153,-4.75)
\psbezier[linecolor=blue, linewidth=0.08, linestyle=dotted, dotsep=0.10583334cm](10.577576,-4.7498407)(10.1776285,-4.7433276)(9.040253,-3.3446226)(10.616654,-2.350159088844206)
\end{pspicture}
}
		\end{center}
		\caption{Superedge $P^{\{u,v\}}$  in $G^{\frac{1}{n}}$ related to the proof of Theorem \ref{G1n}} \label{edge1nodd}
	\end{figure}
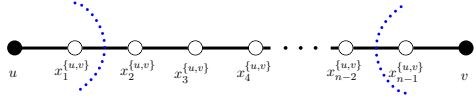

\begin{remark}
The  upper bound in Theorem \ref{G1n} is sharp.  It suffices to consider $G$ as shown in Figure \ref{sharpG1n}. The set of black vertices in $G^{\frac{1}{6}}$, is a co-even dominating set for that with minimum size . By Lemma \ref{lemma}, $\gamma (P_3)=1$. We have $|V(G)| + \gamma(P_{n-3})|E(G)|=6+1(9)=15$. So $\gamma _{coe}(G^{\frac{1}{6}})= |V(G)| + \gamma(P_{3})|E(G)|  $.
\end{remark}

	\begin{figure}
		\begin{center}
			\psscalebox{0.5 0.5}
{
\begin{pspicture}(0,-6.7014422)(19.59423,-0.104326785)
\psdots[linecolor=black, dotsize=0.4](2.5971155,-0.30144227)
\psdots[linecolor=black, dotsize=0.4](0.1971154,-2.7014422)
\psdots[linecolor=black, dotsize=0.4](2.5971155,-5.1014423)
\psline[linecolor=black, linewidth=0.08](2.5971155,-0.30144227)(0.1971154,-2.7014422)(2.5971155,-5.1014423)(2.5971155,-5.1014423)
\psline[linecolor=black, linewidth=0.08](2.5971155,-0.30144227)(6.1971154,-0.30144227)(6.1971154,-0.30144227)
\psline[linecolor=black, linewidth=0.08](2.5971155,-5.1014423)(6.1971154,-5.1014423)(6.1971154,-5.1014423)
\psline[linecolor=black, linewidth=0.08](6.1971154,-0.30144227)(8.5971155,-2.7014422)(8.5971155,-2.7014422)
\psline[linecolor=black, linewidth=0.08](8.5971155,-2.7014422)(6.1971154,-5.1014423)(6.1971154,-5.1014423)
\psdots[linecolor=black, dotsize=0.4](6.1971154,-0.30144227)
\psdots[linecolor=black, dotsize=0.4](8.5971155,-2.7014422)
\psdots[linecolor=black, dotsize=0.4](6.1971154,-5.1014423)
\psline[linecolor=black, linewidth=0.08](2.5971155,-0.30144227)(2.5971155,-5.1014423)(2.5971155,-5.1014423)
\psline[linecolor=black, linewidth=0.08](6.1971154,-0.30144227)(6.1971154,-5.1014423)(6.1971154,-5.1014423)
\psline[linecolor=black, linewidth=0.08](0.1971154,-2.7014422)(8.5971155,-2.7014422)(8.5971155,-2.7014422)
\psdots[linecolor=black, dotsize=0.4](13.397116,-0.30144227)
\psdots[linecolor=black, dotsize=0.4](10.997115,-2.7014422)
\psdots[linecolor=black, dotsize=0.4](13.397116,-5.1014423)
\psline[linecolor=black, linewidth=0.08](13.397116,-0.30144227)(10.997115,-2.7014422)(13.397116,-5.1014423)(13.397116,-5.1014423)
\psline[linecolor=black, linewidth=0.08](13.397116,-0.30144227)(16.997116,-0.30144227)(16.997116,-0.30144227)
\psline[linecolor=black, linewidth=0.08](13.397116,-5.1014423)(16.997116,-5.1014423)(16.997116,-5.1014423)
\psline[linecolor=black, linewidth=0.08](16.997116,-0.30144227)(19.397116,-2.7014422)(19.397116,-2.7014422)
\psline[linecolor=black, linewidth=0.08](19.397116,-2.7014422)(16.997116,-5.1014423)(16.997116,-5.1014423)
\psdots[linecolor=black, dotsize=0.4](16.997116,-0.30144227)
\psdots[linecolor=black, dotsize=0.4](19.397116,-2.7014422)
\psdots[linecolor=black, dotsize=0.4](16.997116,-5.1014423)
\psline[linecolor=black, linewidth=0.08](13.397116,-0.30144227)(13.397116,-5.1014423)(13.397116,-5.1014423)
\psline[linecolor=black, linewidth=0.08](16.997116,-0.30144227)(16.997116,-5.1014423)(16.997116,-5.1014423)
\psline[linecolor=black, linewidth=0.08](10.997115,-2.7014422)(19.397116,-2.7014422)(19.397116,-2.7014422)
\psdots[linecolor=black, dotsize=0.3](12.197115,-1.5014423)
\psdots[linecolor=black, dotsize=0.3](12.197115,-3.9014423)
\psdots[linecolor=black, dotsize=0.3](18.197115,-1.5014423)
\psdots[linecolor=black, dotsize=0.3](18.197115,-3.9014423)
\psdots[linecolor=black, dotsize=0.3](15.177115,-0.28144225)
\psdots[linecolor=black, dotsize=0.3](15.217115,-5.1014423)
\psdots[linecolor=black, dotsize=0.3](15.157115,-2.6814423)
\psdots[linecolor=black, dotsize=0.3](13.397116,-2.3014421)
\psdots[linecolor=black, dotsize=0.3](16.997116,-2.3014421)
\psdots[linecolor=black, dotstyle=o, dotsize=0.3, fillcolor=white](12.997115,-0.70144224)
\psdots[linecolor=black, dotstyle=o, dotsize=0.3, fillcolor=white](12.5971155,-1.1014422)
\psdots[linecolor=black, dotstyle=o, dotsize=0.3, fillcolor=white](11.797115,-1.9014423)
\psdots[linecolor=black, dotstyle=o, dotsize=0.3, fillcolor=white](11.397116,-2.3014421)
\psdots[linecolor=black, dotstyle=o, dotsize=0.3, fillcolor=white](11.397116,-3.1014423)
\psdots[linecolor=black, dotstyle=o, dotsize=0.3, fillcolor=white](11.797115,-3.5014422)
\psdots[linecolor=black, dotstyle=o, dotsize=0.3, fillcolor=white](12.5971155,-4.301442)
\psdots[linecolor=black, dotstyle=o, dotsize=0.3, fillcolor=white](12.997115,-4.7014422)
\psdots[linecolor=black, dotstyle=o, dotsize=0.3, fillcolor=white](14.5971155,-5.1014423)
\psdots[linecolor=black, dotstyle=o, dotsize=0.3, fillcolor=white](15.797115,-5.1014423)
\psdots[linecolor=black, dotstyle=o, dotsize=0.3, fillcolor=white](14.5971155,-0.30144227)
\psdots[linecolor=black, dotstyle=o, dotsize=0.3, fillcolor=white](15.797115,-0.30144227)
\psdots[linecolor=black, dotstyle=o, dotsize=0.3, fillcolor=white](17.797115,-1.1014422)
\psdots[linecolor=black, dotstyle=o, dotsize=0.3, fillcolor=white](17.397116,-0.70144224)
\psdots[linecolor=black, dotstyle=o, dotsize=0.3, fillcolor=white](18.597115,-1.9014423)
\psdots[linecolor=black, dotstyle=o, dotsize=0.3, fillcolor=white](18.997116,-2.3014421)
\psdots[linecolor=black, dotstyle=o, dotsize=0.3, fillcolor=white](18.997116,-3.1014423)
\psdots[linecolor=black, dotstyle=o, dotsize=0.3, fillcolor=white](18.597115,-3.5014422)
\psdots[linecolor=black, dotstyle=o, dotsize=0.3, fillcolor=white](17.797115,-4.301442)
\psdots[linecolor=black, dotstyle=o, dotsize=0.3, fillcolor=white](17.397116,-4.7014422)
\psdots[linecolor=black, dotstyle=o, dotsize=0.3, fillcolor=white](16.337116,-2.7014422)
\psdots[linecolor=black, dotstyle=o, dotsize=0.3, fillcolor=white](17.797115,-2.7014422)
\psdots[linecolor=black, dotstyle=o, dotsize=0.3, fillcolor=white](14.057116,-2.7214422)
\psdots[linecolor=black, dotstyle=o, dotsize=0.3, fillcolor=white](12.5971155,-2.7014422)
\psdots[linecolor=black, dotstyle=o, dotsize=0.3, fillcolor=white](13.397116,-1.1014422)
\psdots[linecolor=black, dotstyle=o, dotsize=0.3, fillcolor=white](16.997116,-1.1014422)
\psdots[linecolor=black, dotstyle=o, dotsize=0.3, fillcolor=white](13.397116,-1.7214422)
\psdots[linecolor=black, dotstyle=o, dotsize=0.3, fillcolor=white](17.017115,-1.7214422)
\psdots[linecolor=black, dotstyle=o, dotsize=0.3, fillcolor=white](13.397116,-3.2814422)
\psdots[linecolor=black, dotstyle=o, dotsize=0.3, fillcolor=white](13.397116,-4.301442)
\psdots[linecolor=black, dotstyle=o, dotsize=0.3, fillcolor=white](16.997116,-3.3014421)
\psdots[linecolor=black, dotstyle=o, dotsize=0.3, fillcolor=white](16.997116,-4.301442)
\psdots[linecolor=black, dotstyle=o, dotsize=0.3, fillcolor=white](14.017116,-0.30144227)
\psdots[linecolor=black, dotstyle=o, dotsize=0.3, fillcolor=white](16.417116,-0.28144225)
\psdots[linecolor=black, dotstyle=o, dotsize=0.3, fillcolor=white](16.337116,-5.1014423)
\psdots[linecolor=black, dotstyle=o, dotsize=0.3, fillcolor=white](14.017116,-5.1014423)
\rput[bl](4.1971154,-6.7014422){\LARGE{$G$}}
\rput[bl](14.997115,-6.7014422){\LARGE{$G^{\frac{1}{6}}$}}
\end{pspicture}
}
		\end{center}
		\caption{Graphs $G$ and $G^{\frac{1}{6}}$ } \label{sharpG1n}
	\end{figure}
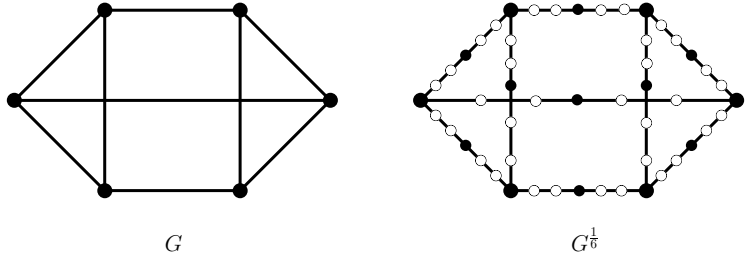

\section{Number of co-even dominating sets of specific graphs}
In this section, we study the number of co-even dominating  sets of some graph classes.
Let ${\cal D}_{coe}(G,i)$ be the family of
co-even  dominating sets of a graph $G$ with cardinality $i$ and let
$d_{coe}(G,i)=|{\cal D}_{coe}(G,i)|$. The generating function for the number of co-even dominating sets of $G$ is $D_{coe}(G,x)=\sum_{i=1}d_0(G,i)x^i$.  In the following, we consider to some special graph classes and compute the number of co-even dominating sets of them and their generating function.  By the definition of co-even dominating set, we have the following easy result:

  \begin{lemma}\label{lemaa2}
If a graph $G$ has no vertices with odd degree, then  ${\cal D}_{coe}(G,i)={\cal D}(G,i)$. 
\end{lemma}
  
  As an immediate result of Lemma \ref{lemaa2}, we have:

  	\begin{proposition}
For the cycle graph $C_n$, we have $d_{coe}(C_n,i)=d(C_n,i)$.
	\end{proposition}

Alikhani et al. computed   $d(C_n,i)$ in \cite{Saeid2}. Now 
 we consider the  path graph $P_n$ and compute $d_{coe}(P_n,i)$. Note that $\gamma _{coe}(P_n)=2+\lceil\frac{n-4}{3}\rceil$ (see \cite{Sha}). We need the following theorem: 
	
	\begin{theorem}\label{path-thmAli}{\rm\cite{Saeid}}
	The number of dominating sets of path graph $P_n$, satisfies  the following recursive relation: 
	\[d(P_n,i)= d(P_{n-1},i-1)+d(P_{n-1},i-2)+d(P_{n-1},i-3).\]
	\end{theorem}

	\begin{theorem}\label{path-thm}
	\begin{itemize}
	\item[(i)]
		For every $n\geq 8$,	the number of co-even dominating sets of path $P_n$ with cardinality $i$, $d_{coe}(P_n,i)$, is: 
	\begin{align*}
	d_{coe}(P_n,i) &= d(P_{n-2},i-2)+2d_{coe}(P_{n-4},i-2)\\
	&\quad +2d_{coe}(P_{n-5},i-2)+d_{coe}(P_{n-6},i-2).
	\end{align*}
	\item[(ii)]
	For every $n\geq 8$, for the generating function  of co-even dominating sets of path $P_n$ we have: 
		\begin{align*}
		D_{coe}(P_n,x)=x^2 \Big( &D(P_{n-2},x)+ 2 D_{coe}(P_{n-4},x) \\
	& +2 D_{coe}(P_{n-5},x) +  D_{coe}(P_{n-6},x) \Big).
		\end{align*}
	\end{itemize}
	\end{theorem}

		\begin{figure}
		\begin{center}
			\psscalebox{0.5 0.5}
			{
\begin{pspicture}(0,-5.48)(15.440476,1.26)
\psdots[linecolor=black, dotsize=0.4](1.7904762,0.57)
\psdots[linecolor=black, dotsize=0.4](3.3904762,0.57)
\psdots[linecolor=black, dotsize=0.4](4.990476,0.57)
\psdots[linecolor=black, dotsize=0.4](6.590476,0.57)
\psdots[linecolor=black, dotsize=0.1](7.7904763,0.57)
\psdots[linecolor=black, dotsize=0.1](8.190476,0.57)
\psdots[linecolor=black, dotsize=0.1](8.590476,0.57)
\psdots[linecolor=black, dotsize=0.4](9.790476,0.57)
\psdots[linecolor=black, dotsize=0.4](11.390476,0.57)
\psdots[linecolor=black, dotsize=0.4](12.990477,0.57)
\psdots[linecolor=black, dotsize=0.4](14.590476,0.57)
\psline[linecolor=black, linewidth=0.08](7.390476,0.57)(1.7904762,0.57)
\psline[linecolor=black, linewidth=0.08](8.990477,0.57)(14.590476,0.57)
\rput[bl](1.6104761,0.97){${v_1}$}
\rput[bl](3.2304761,0.99){${v_2}$}
\rput[bl](4.7704763,0.99){${v_3}$}
\rput[bl](6.390476,1.01){${v_4}$}
\rput[bl](14.390476,0.97){${v_{n}}$}
\rput[bl](12.590476,0.97){${v_{n-1}}$}
\rput[bl](10.990477,0.97){${v_{n-2}}$}
\rput[bl](9.390476,0.97){${v_{n-3}}$}
\psline[linecolor=black, linewidth=0.08](1.3904762,-0.63)(14.990477,-0.63)
\psline[linecolor=black, linewidth=0.08](15.390476,-0.63)(15.390476,-5.43)
\psline[linecolor=black, linewidth=0.08](1.3904762,-1.43)(14.990477,-1.43)
\psline[linecolor=black, linewidth=0.08](1.3904762,-2.23)(14.990477,-2.23)
\psline[linecolor=black, linewidth=0.08](1.3904762,-3.03)(14.990477,-3.03)
\psline[linecolor=black, linewidth=0.08](1.3904762,-3.83)(14.990477,-3.83)
\psline[linecolor=black, linewidth=0.08](1.3904762,-4.63)(14.590476,-4.63)
\psline[linecolor=black, linewidth=0.08](14.990477,-4.63)(14.590476,-4.63)
\psline[linecolor=black, linewidth=0.08](14.990477,-5.43)(1.3904762,-5.43)
\psline[linecolor=black, linewidth=0.08](0.9904762,-0.63)(0.9904762,-5.43)
\psline[linecolor=black, linewidth=0.08](2.5904763,-0.63)(2.5904763,-1.43)
\psline[linecolor=black, linewidth=0.08](13.790476,-0.63)(13.790476,-1.43)
\psline[linecolor=black, linewidth=0.08](1.3904762,-0.63)(0.9904762,-0.63)
\psline[linecolor=black, linewidth=0.08](1.3904762,-1.43)(0.9904762,-1.43)
\psline[linecolor=black, linewidth=0.08](1.3904762,-2.23)(0.9904762,-2.23)
\psline[linecolor=black, linewidth=0.08](1.3904762,-3.03)(0.9904762,-3.03)
\psline[linecolor=black, linewidth=0.08](1.3904762,-3.83)(0.9904762,-3.83)
\psline[linecolor=black, linewidth=0.08](1.3904762,-4.63)(0.9904762,-4.63)
\psline[linecolor=black, linewidth=0.08](1.3904762,-5.43)(0.9904762,-5.43)
\psline[linecolor=black, linewidth=0.08](14.990477,-0.63)(15.390476,-0.63)
\psline[linecolor=black, linewidth=0.08](14.990477,-1.43)(15.390476,-1.43)
\psline[linecolor=black, linewidth=0.08](14.990477,-2.23)(15.390476,-2.23)
\psline[linecolor=black, linewidth=0.08](14.990477,-3.03)(15.390476,-3.03)
\psline[linecolor=black, linewidth=0.08](14.990477,-3.83)(15.390476,-3.83)
\psline[linecolor=black, linewidth=0.08](14.990477,-4.63)(15.390476,-4.63)
\psline[linecolor=black, linewidth=0.08](14.990477,-5.43)(15.390476,-5.43)
\psline[linecolor=black, linewidth=0.08](7.390476,-1.43)(7.390476,-2.23)
\psline[linecolor=black, linewidth=0.08](8.990477,-2.23)(8.990477,-3.03)
\psline[linecolor=black, linewidth=0.08](7.390476,-3.03)(7.390476,-3.83)
\psline[linecolor=black, linewidth=0.08](8.990477,-3.83)(8.990477,-4.63)
\psline[linecolor=black, linewidth=0.08](7.390476,-4.63)(7.390476,-5.43)(7.390476,-5.43)
\psline[linecolor=black, linewidth=0.08](8.990477,-4.63)(8.990477,-5.43)(8.990477,-5.43)
\psline[linecolor=black, linewidth=0.08](12.190476,-1.43)(12.190476,-2.23)
\psline[linecolor=black, linewidth=0.08](4.1904764,-2.23)(4.1904764,-3.03)(4.1904764,-3.03)
\psline[linecolor=black, linewidth=0.08](10.590476,-3.03)(10.590476,-3.83)
\psline[linecolor=black, linewidth=0.08](5.7904763,-3.83)(5.7904763,-4.63)
\psdots[linecolor=black, dotsize=0.4](1.7904762,-1.03)
\psdots[linecolor=black, dotsize=0.4](1.7904762,-1.83)
\psdots[linecolor=black, dotsize=0.4](1.7904762,-2.63)
\psdots[linecolor=black, dotsize=0.4](1.7904762,-3.43)
\psdots[linecolor=black, dotsize=0.4](1.7904762,-4.23)
\psdots[linecolor=black, dotsize=0.4](1.7904762,-5.03)
\psdots[linecolor=black, dotsize=0.4](14.590476,-1.03)
\psdots[linecolor=black, dotsize=0.4](14.590476,-1.83)
\psdots[linecolor=black, dotsize=0.4](14.590476,-2.63)
\psdots[linecolor=black, dotsize=0.4](14.590476,-3.43)
\psdots[linecolor=black, dotsize=0.4](14.590476,-4.23)
\psdots[linecolor=black, dotsize=0.4](14.590476,-5.03)
\psdots[linecolor=black, dotsize=0.4](12.990477,-1.83)
\psdots[linecolor=black, dotsize=0.4](9.790476,-2.63)
\psdots[linecolor=black, dotsize=0.4](11.390476,-3.43)
\psdots[linecolor=black, dotsize=0.4](9.790476,-4.23)
\psdots[linecolor=black, dotsize=0.4](9.790476,-5.03)
\psdots[linecolor=black, dotsize=0.4](6.590476,-5.03)
\psdots[linecolor=black, dotsize=0.4](4.990476,-4.23)
\psdots[linecolor=black, dotsize=0.4](6.590476,-3.43)
\psdots[linecolor=black, dotsize=0.4](3.3904762,-2.63)
\psdots[linecolor=black, dotsize=0.4](6.590476,-1.83)
\psdots[linecolor=black, dotstyle=o, dotsize=0.4, fillcolor=white](3.3904762,-1.83)
\psdots[linecolor=black, dotstyle=o, dotsize=0.4, fillcolor=white](4.990476,-1.83)
\psdots[linecolor=black, dotstyle=o, dotsize=0.4, fillcolor=white](11.390476,-2.63)
\psdots[linecolor=black, dotstyle=o, dotsize=0.4, fillcolor=white](12.990477,-2.63)
\psdots[linecolor=black, dotstyle=o, dotsize=0.4, fillcolor=white](12.990477,-3.43)
\psdots[linecolor=black, dotstyle=o, dotsize=0.4, fillcolor=white](11.390476,-4.23)
\psdots[linecolor=black, dotstyle=o, dotsize=0.4, fillcolor=white](12.990477,-4.23)
\psdots[linecolor=black, dotstyle=o, dotsize=0.4, fillcolor=white](12.990477,-5.03)
\psdots[linecolor=black, dotstyle=o, dotsize=0.4, fillcolor=white](11.390476,-5.03)
\psdots[linecolor=black, dotstyle=o, dotsize=0.4, fillcolor=white](3.3904762,-4.23)
\psdots[linecolor=black, dotstyle=o, dotsize=0.4, fillcolor=white](3.3904762,-3.43)
\psdots[linecolor=black, dotstyle=o, dotsize=0.4, fillcolor=white](4.990476,-3.43)
\psdots[linecolor=black, dotstyle=o, dotsize=0.4, fillcolor=white](3.3904762,-5.03)
\psdots[linecolor=black, dotstyle=o, dotsize=0.4, fillcolor=white](4.990476,-5.03)
\rput[bl](0.31047618,-1.2442857){$(I)$}
\rput[bl](0.13333333,-2.0967858){$(II)$}
\rput[bl](0.0,-2.8967857){$(III)$}
\rput[bl](0.038095236,-3.639643){$(IV)$}
\rput[bl](0.1904762,-4.4586906){$(V)$}
\rput[bl](0.07619047,-5.2205954){$(VI)$}
\end{pspicture}
}
		\end{center}
		\caption{ Making co-even dominating sets of $P_n$ related to the proof of Theorem \ref{path-thm}} \label{pathn}
	\end{figure}
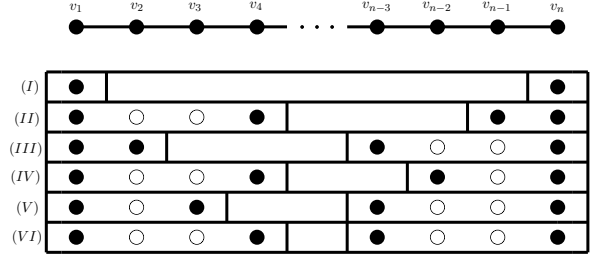

\begin{proof}
	\begin{itemize}
	\item[(i)]
Since we have two vertices with odd degree, then we need these vertices in our co-even dominating set. Suppose that this set is $S$. We consider the following cases (See figure \ref{pathn}):
\begin{itemize}
	\item[(I)]
	We find a dominating set  with cardinality $i-2$ for the rest of vertices which is a path of order ${n-2}$. So, in this case, we have $d(P_{n-2},i-2)$ co-even dominating sets. Note that in this case, at least one of the vertices $v_2$ or $v_3$ and at least  one of the vertices $v_{n-1}$ or $v_{n-2}$ are in $S$.
	\item[(II)]
	$v_2,v_3\notin S$ and $v_{n-1}\in S$. In this case,	we should have $v_4\in S$. By considering the path consist of vertices $v_4,v_5,v_6,\ldots,v_{n-1}$, we have a path of order $n-4$, and since $v_4,v_{n-1}\in S$, then finding a dominating set for this path is like finding a co-even dominating set for $P_{n-4}$. So, in this case, we have $d_{coe}(P_{n-4},i-2)$ co-even dominating sets. 
	\item[(III)]
$v_{n-2},v_{n-1}\notin S$ and $v_{2}\in S$. In this case,	we should have $v_{n-3}\in S$.	By similar argument as case (II), we have 		$d_{coe}(P_{n-4},i-2)$ co-even dominating sets.
	\item[(IV)]
	$v_2,v_3,v_{n-1}\notin S$ and $v_{n-2}\in S$. In this case,	we should have $v_4\in S$. By considering the path consist of vertices $v_4,v_5,v_6,\ldots,v_{n-2}$, we have a path of order $n-5$, and since $v_4,v_{n-2}\in S$, then finding a dominating set for this path is like finding a co-even dominating set for $P_{n-5}$. So, in this case, we have $d_{coe}(P_{n-5},i-2)$ co-even dominating sets. 				
	\item[(V)]
$v_{2},v_{n-2},v_{n-1}\notin S$ and $v_{3}\in S$. In this case,	we should have $v_{n-3}\in S$.	By similar argument as case (IV), we have 		$d_{coe}(P_{n-5},i-2)$ co-even dominating sets.					
	\item[(VI)]
	$v_2,v_3,v_{n-1},v_{n-2}\notin S$. In this case,	we should have $v_4,v_{n-3}\in S$. By considering the path consist of vertices $v_4,v_5,v_6,\ldots,v_{n-3}$, we have a path of order $n-6$, and since $v_4,v_{n-3}\in S$, then finding a dominating set for this path is like finding a co-even dominating set for $P_{n-6}$. So, in this case, we have $d_{coe}(P_{n-6},i-2)$ co-even dominating sets. 	
\end{itemize}
Therefore, we have the result.
	\item[(ii)]
It follows from Part (i).
\qed
	\end{itemize}
\end{proof}

It is easy to find $d_{coe}(P_n,j)$ for $1\leq n\leq 7$. By using Theorem \ref{path-thm}, we obtain $d_{coe}(P_n,j)$ for $1\leq n\leq 12$ as shown in Table 1.

\[
\begin{footnotesize}
\small{
	\begin{tabular}{r|lcrrrcccccccc}
	$j$&$1$&$2$&$3$&$4$&$5$&$6$&$7$&$8$&$9$&$10$&$11$&$12$\\[0.3ex]
	\hline
	$n$&&&&&&&&&&&&\\
	$1$&1&&&&&&&&&&&\\
	$2$&0&1&&&&&&&&&&\\
	$3$&0&1&1&&&&&&&&&\\
	$4$&0&1&2&1&&&&&&&&\\
	$5$&0&0&3&3&1&&&&&&&\\
	$6$&0&0&2&6&4&1&&&&&&\\
	$7$&0&0&1&7&10&5&1&&&&&\\
	$8$&0&0&0&6&16&15&6&1&&&&\\
	$9$&0&0&0&3&19&30&21&7&1&&&\\
	$10$&0&0&0&1&16&46&50&28&8&1&&\\
	$11$&0&0&0&0&10&49&90&76&36&9&1&\\
	$12$&0&0&0&0&4&45&126&161&112&45&10&1\\
	\end{tabular}}
\end{footnotesize}
\]
\begin{center}
	\noindent{Table 1.} $d_{coe}(P_n,j)$ The number of co-even dominating sets of $P_n$ with cardinality $j$.
\end{center}

Here, we consider the number of co-even dominating sets of complete bipartite graphs $K_{n,m}$. If $n$ and $m$ are odd, then we need all the vertices in our co-even dominating set. If $n$ is odd and $m$ is even, then we need $m+1$ vertices to have co-even dominating set. If $n$ and $m$ are even, then we need two vertices for our co-even dominating set. By an easy argument we have:

\begin{proposition}\label{bipartite-gen}
For the number of co-even dominating sets of complete bipartite graphs $K_{n,m}$ with cardinality $i$ where $0<i\leq n+m$, 
\begin{itemize}
\item[(i)]
$d_{coe}(K_{n,m},i)=0$ if $n$ and $m$ are odd and $i<n+m$, and 
$$d_{coe}(K_{n,m},n+m)=1.$$
\item[(ii)]
$d_{coe}(K_{n,m},i)=0$ if $n$ is odd and $m$ is even and $i<m+1$, and  for $ m+1 \leq i\leq n+m$ we have 
$$d_{coe}(K_{n,m},i)= {n \choose i-m}.$$
\item[(iii)]
$d_{coe}(K_{n,m},1)=0$ if $n$ and $m$ are even, and  for $ 2 \leq i\leq n+m$, where $i=t+s$ such that $1\leq t\leq n$ and $1\leq s\leq m$  we have 
$$d_{coe}(K_{n,m},i)= \displaystyle\sum_{t+s=i}^{} {n \choose t}{m \choose s}.$$
\end{itemize}
\end{proposition}

As an immediate result of Proposition \ref{bipartite-gen}, we have:

\begin{corollary}
	The generating function for the number of   of co-even dominating sets of complete bipartite graphs $K_{n,m}$ is 
\begin{itemize}
\item[(i)]
$D_{coe}(K_{n,m},x)=x^{n+m}$, if $n$ and $m$ are odd.
\item[(ii)]
$D_{coe}(K_{n,m},x)=\sum_{i=m+1}^{n+m}{n \choose i-m}x^{i}$, if $n$ is odd and $m$ is even.
\item[(iii)]
$D_{coe}(K_{n,m},x)=\sum_{i=2}^{n+m}\sum_{t+s=i}^{} {n \choose t}{m \choose s}x^{i}$, if $n$ and $m$ are even.
\end{itemize}
\end{corollary}

Now we consider the number of co-even dominating sets of complete graphs $K_{n}$. If $n$ is even, then the degree of each vertex is odd and we need all vertices in our co-even dominating set, but if $n$ is odd, then we need only one vertex for our co-even dominating set. Hence, we have the following result:

\begin{proposition}\label{complete-gen}
For the number of co-even dominating sets of complete graphs $K_{n}$ with cardinality $i$ where $0<i\leq n$, 
\begin{itemize}
\item[(i)]
$d_{coe}(K_{n},i)=0$ if $n$ is even and $i<n$, and 
$d_{coe}(K_{n},n)=1$.
\item[(ii)]
If $n$ is odd, then 
$d_{coe}(K_{n},i)= {n \choose i}$.
\end{itemize}
\end{proposition}

We end this section by an immediate result of Proposition \ref{complete-gen}:

\begin{corollary}
	The generating function for the number of   of co-even dominating sets of complete graphs $K_{n}$ is 
\begin{itemize}
\item[(i)]
$D_{coe}(K_{n},x)=x^{n}$, if $n$ is even.
\item[(ii)]
$D_{coe}(K_{n},x)=\sum_{i=1}^{n}{n \choose i}x^{i}$, if $n$ is odd.
\end{itemize}
\end{corollary}

\section{Conclusions}

In this paper,  we obtained some lower and upper bounds of co-even domination number of graphs which constructed  by  removing edges between all neighbours of a vertex, and show that these bounds are sharp. Then, we presented sharp upper bound for co-even domination number of a $k$-subdivision of a graph. Also we introduced the concept of the generating function for the number of co-even dominating sets of a graph and computed that for specific graphs.  Future topics of interest for future research include the following suggestions:

\begin{itemize}
\item[(i)]
Finding generating function for the number of co-even dominating sets of other graph classes.
\item[(ii)]
Finding lower bound for co-even domination number of a $k$-subdivision of a graph.
\end{itemize}

\section{Acknowledgements} 
	
	The  author would like to thank the Research Council of Norway and Department of Informatics, University of
	Bergen for their support.

\end{document}